\documentclass[12pt,reqno]{amsart}

\usepackage{epsfig}
\usepackage{color}
\usepackage{amsmath, amsthm, amsfonts, amssymb, mathrsfs}
\usepackage{graphicx}
\usepackage{enumerate}

\numberwithin{equation}{section}

\newcommand{\wh}{\widehat}

\newcommand{\C}{{\mathbb C}}

\newtheorem{theo}{{\sc \bf Theorem}}[section]
\newtheorem{cor}[theo]{{\sc \bf Corollary}}
\newtheorem{lem}[theo]{{\sc \bf Lemma}}
\newtheorem{prop}[theo]{{\sc \bf Proposition}}

\theoremstyle{definition}
\newtheorem{defin}{Definition}[section]

\begin{document}

\title{A $P$-Adic Spectral Triple}

\author{Slawomir Klimek}
\address{Department of Mathematical Sciences,
Indiana University-Purdue University Indianapolis,
402 N. Blackford St., Indianapolis, IN 46202, U.S.A.}
\email{sklimek@math.iupui.edu}

\author{Matt McBride}
\address{Department of Mathematics,
University Of Oklahoma,
601 Elm Ave., Norman, OK 73019,  U.S.A.}
\email{mmcbride@math.ou.edu }

\author{Sumedha Rathnayake}
\address{Department of Mathematical Sciences,
Indiana University-Purdue University Indianapolis,
402 N. Blackford St., Indianapolis, IN 46202, U.S.A.}
\email{srathnay@iupui.edu}

\date{\today}

\begin{abstract}
We construct a spectral triple for the C$^*$-algebra of continuous functions on the space of $p$-adic integers by using a rooted tree obtained from coarse-grained approximation of the space, and the forward derivative on the tree. Additionally, we verify that our spectral triple satisfies the properties of a compact spectral metric space, and we show that the metric on the space of $p$-adic integers induced by the spectral triple is equivalent to the usual $p$-adic metric.
\end{abstract}

\maketitle
\section{Introduction}

This paper studies the space of $p$-adic integers from the point of view of noncommutative geometry.
To motivate our interest in this problem we start with a short overview of the concept of a noncommutative (quantum) space.
Simply put, unital C$^*$-algebras will be called compact quantum spaces. This is due to the Gelfand-Naimark Theorem, which states that commutative unital C$^*$-algebras are precisely the algebras of continuous functions on compact topological spaces. In this paper we will only consider unital algebras.
\smallskip

In the spirit of noncommutative geometry we look at extra structures on  C$^*$-algebras corresponding to geometrical notions on topological spaces. In a series of papers \cite{R1}, \cite{R2}, \cite{R3}, and inspired by \cite{C2}, M. Rieffel proposed the concept of a compact quantum metric space based on the following observations.
If $(X, \rho)$ is a compact metric space  and $\phi \in C(X)$, the space of continuous functions on $X$, let $L(\phi)$ be the Lipschitz constant of $\phi$ defined by:
\begin{equation*}
L(\phi) = \sup_{x\ne y} \frac{|\phi (x) - \phi (y)|}{\rho(x, y)} \in [0, \infty].
\end{equation*}
Then $L$ is a seminorm on $C(X)$, and the data $(C(X) , L)$ is sufficient to recover $(X, \rho)$, via Gelfand-Naimark Theorem and the formula:
\begin{equation*}
\rho(x,y)=\sup_{\phi}\{  |\phi(x)-\phi(y)|,\  L(\phi )\leq1\}.
\end{equation*}
So the notion of a metric can be reformulated in terms of the C$^*$-algebra $C(X)$ through a seminorm satisfying some additional properties, see \cite{R3} for the precise definition of a compact quantum metric space.
\smallskip

There are many Lipschitz like seminorms on C$^*$-algebras, but we will be particularly interested in those coming from spectral triples \cite{C}, a concept which arose from filtering central properties of Dirac type operators on manifolds. There are several versions of the definition of the spectral triple in the literature, appropriate for different contexts. We will use the following:

\begin{defin} A spectral triple for a C$^*$-algebra $A$ is a triple $(\mathcal A, \mathcal H,\mathcal D)$ where $\mathcal H$ is a Hilbert space on which $A$ is represented by bounded operators (i.e. there exists  a $\ast$-homomorphism $\Pi: A \to \mathcal B( \mathcal H)$),  $\mathcal A$ is a dense $\ast-$subalgebra of $A$, and $\mathcal D$ is a self-adjoint operator (typically unbounded) on $\mathcal H$ satisfying: 
\begin{enumerate}
\item for every $a\in \mathcal A$ the commutator $[\mathcal D,\Pi(a)]$ is bounded,
\item $(1+ \mathcal D^2)^{-1/2}$ is a compact operator.
\end{enumerate}
\end{defin}

Additionally, a spectral triple is called even if there is a $\mathbb Z/2\mathbb Z$ grading on $\mathcal H$ with respect to which the representation $\Pi$ is even, while the operator $\mathcal D$ is odd with respect to the grading. This will be the case in our paper. Spectral triples with faithful representation of $A$, called unbounded Fredholm modules, have been used initially by A. Connes for the cyclic cohomology of the noncommutative space defined by $A$ \cite{C}. He also showed in a particular example of the Dirac operator on a compact Riemannian manifold \cite{C2} that spectral triples contain information on the metric structure of the space via Connes metric formula described below.

Returning to the notion of compact quantum metric spaces, we follow a more recent paper \cite{BMR}, where the authors propose the following terminology. 

\begin{defin}\label{CQMS} Starting with a spectral triple $(\mathcal A, \mathcal H, \mathcal D)$ for a C$^*$-algebra $A$ we define $L_{\mathcal D}(a):= ||\;[\mathcal D, \pi(a)]\;||$. If the conditions $(1)-(3)$ below are satisfied, then the pair $(A, L_{\mathcal D})$ is called a compact spectral metric space:
\begin{enumerate}
\item The representation of $A$ in $\mathcal H$ is non-degenerate, i.e. $A\mathcal H=\mathcal H$.
\item The commutant $A_{\mathcal D}' = \{a \in A:\ [\mathcal D,\pi(a)] = 0\}$ is trivial, i.e. $A_{\mathcal D}' = \C I$.
\item The image of the Lipschitz ball $B_{\mathcal D} = \{a \in A :\  L_{\mathcal D}(a) \leq 1\}$ is precompact in $A/A_{\mathcal D}'$.
\end{enumerate}
\end{defin}

It follows from the work of M. Rieffel that $(A, L_{\mathcal D})$ is a spectral metric space if and only if the Connes metric, defined on the state space of $A$ via the formula 
\begin{equation*}
d_C(\phi, \psi) = \sup_a\{|\phi(a) - \psi(a)| :\  L_{\mathcal D}(a) \leq 1\}
\end{equation*}
 is well defined and the topology induced by $d_C$ is equivalent to the weak* -topology.  In the commutative case this extension of a metric from the compact space $X$ to its set of probability measures, in which case is the the space of states for $C(X)$, had been defined and studied before by Kantorovich and Rubinstein \cite{KR}.
\smallskip

It is important to keep in mind that there is more geometrical information encoded in a spectral triple. For example, for a compact spin Riemannian manifold one can retrieve its smooth structure, its Riemannian metric and other properties directly from its standard Dirac operator \cite{C1}. Consequently any closed Riemannian manifold endowed with a spin structure can be recovered from the spectral triple of its algebra of continuous functions and the Dirac operator acting on $L^2$ spinors, via the reconstruction theorem of Connes \cite{C1} and Lord et al.\cite{LRV}.
\smallskip

The goal of this paper is to present a graph theoretic construction of a spectral triple for the C$^*$-algebra of continuous functions on the space of $p$-adic integers. Several constructions of similar spectral triples on Cantor sets can be found in the literature, including the original proposal in Connes book \cite{C} (see also \cite{CI}). Our starting point is the same as in \cite{P} and \cite{BP} in the sense that we are using the rooted tree whose vertices are balls, which can be considered as a coarse grained approximation of the space of $p$-adic integers. We then use the ``forward" derivative on this tree to construct the spectral triple, leading to a more elaborate Dirac-type operator. In fact in \cite{P} and \cite{BP} the operator is essentially diagonal in the standard basis of $l^2$ functions on the vertices of the graph. The authors then use a clever construction of a representation of the Lipschitz functions in the Hilbert space to reconstruct the metric on the Cantor set. In our construction we simply consider a more straightforward, diagonal representation of the algebra of functions. However, since our techniques use Fourier transform, they are currently not available for general Cantor sets discussed in \cite{P} and \cite{BP}.
\smallskip

Additionally, introducing and studying noncommutative notions like spectral triples perhaps allows to view the Cantor set of $p$-adic integers as more than a mere metric space, namely as some sort of a differentiable space. Moreover there is a number theoretic angle to considering operators related to $p$-adic numbers, as they may lead to interesting spectral functions. 
\smallskip

The paper is organized as follows. The results are presented in four sections. First we describe the ``$p$-adic tree" associated to the space of $p$-adic integers via Michon's correspondence \cite{BP}. We also review some basics of Fourier analysis in the space of $p$-adic integers. In the following section we analyze the operator $D$ of ``forward" derivative on the $p$-adic tree which will be the base of our construction of the spectral triple. We investigate its distributional properties and its inverse in the Hilbert space of the weighted $l^2$ functions on the vertices of the $p$-adic tree. The key tool we are using here is the $p$-adic Fourier transform which reduces $D$ to Jacobi type operators similar to those studied in \cite{KM}. The spectrum of $D^*D$ is studied in \cite{KRS}.
In section \ref{sec4}, we use the operator $D$ to construct a spectral triple for the algebra of continuous functions on the set of $p$-adic integers, and verify the relevant properties. In the last section we study the metric on $p$-adic integers, induced by our spectral triple and show in particular that it is equivalent to the usual $p$-adic metric. We also verify that our spectral triple gives a compact spectral metric space structure on $C(\mathbb Z_p)$.

\section{The $p$-adic tree}

\subsection{} The main object we study in this paper is the Cantor set given by the ring of $p$-adic integers, $\mathbb Z_p$, equipped with the $p$-adic metric defined by $\rho_p(x,y)=|x-y|_p$. Recall that $\mathbb Z_p$ is the subset of the set of all $p$-adic numbers $\mathbb Q_p$ whose $p$-adic norm is less than or equal to one. Equivalently, $\mathbb Z_p$ is the completion of the set of integers, $\mathbb Z$, with respect to the $p$-adic metric $\rho_p$. Because $\mathbb Z_p$ is totally disconnected the range of $\rho_p$ is countable and, other than zero, consists of numbers of the form $p^{-n}$ where $n$ is a nonnegative integer. 
\smallskip

It is useful to recall the $p$-adic representation of an $x\in \mathbb Q_p$ as a convergent series 
\begin{equation*}
x=\sum_{n=\alpha}^\infty x_n p^n, \;\; \alpha \in \mathbb Z
\end{equation*}
with unique $x_n\in\{0,1,2,\ldots,p-1\}$ and  $x_\alpha\ne 0$. The $p$-adic norm of $x$ is then defined as $|x|_p=p^{-\alpha}$. Notice that for any $x\in \mathbb Z_p$ we will always have $\alpha\geq 0$. The fractional part of a $p$-adic number $x\in \mathbb Q_p$, denoted $\{x\}$, is defined to be 
\begin{equation*}
\{x\}:=\begin{cases}
\sum\limits_{n=\alpha}^{-1} x_n p^n & \textrm { if } \alpha\leq -1\\
0 & \textrm { otherwise}
\end{cases}
\end{equation*}
\bigskip

\subsection{}

We will now briefly review some basics of harmonic analysis on $\mathbb Z_p$. A more elaborate treatment of this subject can be found in \cite{HR} and \cite{VVZ}.
\smallskip

 A map $\chi_a: \mathbb Q_p \rightarrow \mathbb C$ defined by 
\begin{equation*}
\chi_a(x)=e^{2\pi i \{ax\}}
\end{equation*}
is a character on $\mathbb Q_p$. In fact, it can be shown that every character is of the form  $\chi_a(x)=e^{2\pi i \{ax\}}$ for a unique $a \in \mathbb Q_p$. Two such characters $\chi_a(x)$ and $\chi_{a'}(x)$ coincide on $\mathbb Z_p$ if and only if $(a'-a)\in \mathbb Z_p$. Consequently we see that the dual groups $\wh {\mathbb Q_p}$, $\widehat{\mathbb Z_p}$ of $\mathbb Q_p, \mathbb Z_p$ are,
\begin{equation*}
\wh{\mathbb Q}_p=\mathbb Q_p, \;\; \wh{\mathbb Z}_p=\mathbb Q_p/\mathbb Z_p.
\end{equation*}

The discrete group $\wh{\mathbb Z_p}$, called the Pr\"ufer group has many presentations. It can be given the structure of an inductive limit of groups $\mathbb Z/p^n\mathbb Z$. On the other hand, it can also be identified with a group of roots of unity 
\begin{equation*}
\{z\in\mathbb Z: \exists n \in \{0,1,2,\ldots\} \textrm{ with } z^{p^n}=1\}. 
\end{equation*}
Every such root of unity can be written uniquely as
$z=e^{\frac{2\pi ik}{p^n}}$ with $p\nmid  k$. So, the corresponding character on $\mathbb Z_p$ is
\begin{equation}\label{nkchar}
\chi_{n,k}(x)=e^{2\pi i\{\frac{kx}{p^n}\}}.
\end{equation}
Consequently,  $\wh{\mathbb Z}_p\cong \{e^{2\pi i\{\frac{kx}{p^n}\}}: n \in \{0,1,2,\ldots\}, p\nmid k \in \mathbb Z, x\in \mathbb Z_p\}$.
\smallskip

A function $\phi: \mathbb Z_p \rightarrow \mathbb C$ is called locally constant if for every $x\in \mathbb Z_p$ there exists a neighborhood $U_x$ of $x$ such that $\phi$ is constant on $U_x$. Any such function is constant on every ball of sufficiently small radius. We let $n(\phi)$ be the smallest positive integer $n$ such that $\phi$ is constant on every ball of radius $p^{-n}$. The space of locally constant functions on $\mathbb Z_p$ is called the space of test functions and will be denoted by $\mathcal E(\mathbb Z_p)$. The convergence in $\mathcal E(\mathbb Z_p)$ is defined in the following way.
\smallskip

If $\phi_k\in\mathcal E(\mathbb Z_p)$ is a sequence of test functions then $\phi_k\to 0$ as $k\to\infty$ if,
\begin{enumerate}
\item the functions $\phi_k$ are uniformly constant, i.e. the sequence of numbers  $\{n(\phi_k)\}$ is bounded,
\item $\phi_k(x) \xrightarrow\ 0$ as $k\to\infty$ for every ${x\in \mathbb Z_p}$.
\end{enumerate}

The space of distributions of $\mathbb Z_p$, denoted $\mathcal E^*(\mathbb Z_p)$, is the space of linear functionals on $\mathcal E(\mathbb Z_p)$ equipped with the weak * -topology. It follows from the definition of the topology on $\mathcal E(\mathbb Z_p)$ that every linear functional $T\in \mathcal E^*(\mathbb Z_p)$ is automatically continuous.
A proof of this result can be found in \cite{VVZ}.

Since $(\mathbb Z_p,+)$ is a compact abelian group there exists a Haar measure on $\mathbb Z_p$ which will be denoted by $d_px$. This Haar measure satisfies $d_p(xa)=|a|_p d_px$ and is normalized so that $\int_{\mathbb Z_p}\,d_px=1$.
\smallskip

\subsection{}

The Fourier transform of a test function $\phi$ on $\mathbb Z_p$ is the function $\wh\phi$ on the classes $\mathbb Q_p/\mathbb Z_p$ given by
\begin{equation*}
\wh\phi([a])=\int_{\mathbb Z_p}\phi(x)\overline{\chi_a(x)}\,d_px.
\end{equation*}
It can be easily verified that $\int_B{\chi_a(x)}\,d_px=0$ for any sufficiently small ball $B$ with radius $p^{-n}$ such that $p^n\leq |a|_p$. Consequently, only finite number of Fourier coefficients of a locally constant function are nonzero. The Fourier transform gives an isomorphism between $\mathcal E(\mathbb Z_p)$ and $\mathcal E(\wh{\mathbb Z}_p)$, the space of compactly supported functions on $\wh{\mathbb Z}_p$, which in our case means the space of functions which are zero almost everywhere. The inverse Fourier transform is given by:
\begin{equation*}
\phi(x)=\sum_{[a]\in \wh{\mathbb Z}_p}\wh\phi([a])\chi_a(x).
\end{equation*}

If $T\in \mathcal E^*(\mathbb Z_p)$ is a distribution on $\mathbb Z_p$ then its Fourier transform is the function $\wh T$ on $\wh{\mathbb Z_p}$ given by
\begin{equation*}
\wh T([a])=T\left(\overline{\chi_a(x)}\right).
\end{equation*}
The Fourier transform is an isomorphism between $\mathcal E^*(\mathbb Z_p)$ and $\mathcal E^*(\wh{\mathbb Z}_p)$, the space of all functions on $\wh{\mathbb Z}_p$. The inverse Fourier transform of a distribution is given by:
\begin{equation*}
T=\sum_{[a]\in \wh{\mathbb Z}_p}\wh T([a])\chi_a(x).
\end{equation*}
The formal sum above makes distributional sense because test functions on $\wh{\mathbb Z}_p$ are non zero only at a finite number of points.

As usual, the distributional Fourier transform $T\mapsto \wh T$ gives a Hilbert space isomorphism
\begin{equation*}
L^2({\mathbb Z}_p, d_px) \cong  l^2(\wh{\mathbb Z}_p).
\end{equation*}
\smallskip

\subsection{}

Via Michon's correspondence \cite{M} we can associate a weighted, rooted tree $\{V,E\}$ to the Cantor metric space $(\mathbb Z_p,\rho_p)$. Here $V$ and  $E$ are the sets of vertices and edges of the tree respectively. The tree is constructed in the following way. The vertices of the tree are the balls in $(\mathbb Z_p,\rho_p)$. The set of vertices can be written  as $V=\bigcup_{n=0}^\infty V_n$ where $V_n=\{\textrm{balls of diameter }p^{-n}\}$. 

The set of edges $E$ also has a decomposition $E=\bigcup_{n=1}^\infty E_n$ where if $e\in E_n$ then $e=(v,v')$ where $v\in V_n$ and $v'\in V_{n+1}$, and  $v'\subset v$. Then we say that there is an edge (undirected) between vertices $v$ and $v'$. Since $\mathbb Z_p$ is a compact set, for each $n$ the number of balls with diameter $p^{-n}$ and the degree of each vertex $v\in  V$ are finite.

The root of the weighted rooted tree corresponding to $(\mathbb Z_p,\rho_p)$ is $\mathbb Z_p$, the unique ball of radius one.  To conveniently parametrize the vertices we need the following observation.
\smallskip

\begin{prop}
Every ball of radius $p^{-n}$ contains exactly one integer $k$ such that $0\leq k<p^n$. 
\end{prop}
\begin{proof}
To prove the uniqueness we notice that for two nonnegative integers $k_1, k_2$ to be in the same ball of radius $p^{-n}$ their difference $k_1-k_2$ would have to be divisible by $p^{n}$ and the inequality $0 \leq k_1, k_2 < p^n$ prevents that from happening. To construct such an integer in any ball simply pick one element $x=\sum_{i=0}^\infty x_i p^i$ in it, and chop off all the terms with powers of $p$ greater than or equal to $n$. The resulting integer
$\sum_{i=0}^{n-1} x_i p^i$ is in the same ball as $x$ and is less then $p^{n}$. 
\end{proof}

Thus there is a one-to-one correspondence between the set of vertices $V_n$ and the set of nonnegative integers less than $p^n$, i.e. we write
\begin{equation}\label{vlabel}
V=\{(n,k)|\ n=0,1,2,\ldots,\ 0\leq k<p^n\}.
\end{equation}

The ball $(n+1,k')$ is contained in the ball $(n,k)$ ( i.e., there is an edge between vertex $(n,k)$ and $(n+1,k')$) if and only if $p^n$ divides the difference $k'-k$. Thus we see that, from each vertex $(n,k)$ comes out exactly $p$ edges connecting $(n,k)$ to $(n+1,k+ip^n)$ for $i=0,1,\ldots, p-1$.
\smallskip

If $e$ is an edge $e=(v,v')$ where $v$ is in $V_n$ and $v'$ is in $V_{n+1}$ then the vertex $v'$ uniquely determines $e$. Thus the set of edges can be parametrized by assigning the coordinates of $v'$ to $e$ . Consequently, there is a natural one-to-one correspondence between the set of vertices minus the root, and the edges. From now on we will refer to this tree as the $p$-adic tree.

The weight function $\omega: V\to\mathbb R^+$ is given by $\omega(v)=\textrm{diameter}(v)$, so that if $v\in V_n$ then $\omega(v)=p^{-n}$.
\smallskip

\subsection{}

Notice that the decomposition $V=\bigcup_{n=0}^\infty V_n$ of the set of vertices implies that any complex valued function $f$ on $V$ is given by a sequence of complex valued functions $\{f_n\}$, where the domain of $f_n$ is $V_n$. In this paper we prefer to work with the latter notation $\{f_n\}$. Next, we can naturally identify the set $V_n$ with $p^n$ elements, with $\mathbb Z/p^n\mathbb Z$. Thus we can give $V_n$ the structure of a finite group, using addition modulo $p^n$. 
\smallskip

The key tool we will be using in the analysis below is the following Fourier transform on the space $\mathcal E^*(V)$ of complex valued functions on vertices of the $p$-adic tree. The transform is just the usual discrete Fourier transform on each $V_n$ where $V_n$ is viewed as $\mathbb Z/p^n\mathbb Z$, and it is given by
\begin{equation}\label{formula_for_fn}
\wh f_n(l)=\frac{1}{p^n}\sum_{k=0}^{p^n-1}f_n(k)e^{-2\pi i\frac{kl}{p^n}}, \ 0\leq l<p^n.
\end{equation}

One can easily show that 
\begin{equation*}
\sum_{0\leq s<p^j}e^\frac{-2\pi iks}{p^j}=\begin{cases}
0 & \textrm{ if }\ p^j\nmid k \\ 
p^j & \textrm{ if }\ p^j\mid k.
\end{cases}
\end{equation*}
This implies the following Fourier inversion formula.
\begin{equation}\label{FIof_f}
f_n(k)=\sum_{0\leq l<p^n}\wh f_n(l)e^{\frac{2\pi ikl}{p^n}}
\end{equation}

To distinguish between the domain and the range of the Fourier transform it will be convenient to introduce the ``dual" tree $\wh V$. However, since $\mathbb Z/p^n\mathbb Z$ is self-dual we set $\wh V=V$ and think of the Fourier transform on the $p$-adic tree as an isomorphism between $\mathcal E^*(V)$ and $\mathcal E^*(\wh V)$.
It is also useful to recall the Parseval's identity;
\begin{equation}\label{parsev}
\sum_{0\leq k<p^n}|f_n(k)|^2\,p^{-n}=\sum_{0\leq l<p^n}|\wh f_n(l)|^2
\end{equation}
which implies that the Fourier transform gives an isomorphism between the weighted Hilbert space $\ell^2(V, \omega)$ and the unweighted space $\ell^2(\wh V)$.
\bigskip

\section{The operator}

For vertices $v, v'\in V$, we will write $v' \sqsubset v$  if $v' \subset v$ and there is an edge between $v$ and $v'$. 
For a complex-valued function $f\in \mathcal E^*(V)$ on the set of vertices $V$ we define what amounts to the forward derivative on the tree $\{V, E\}$:

\begin{equation*}
Df(v)=\frac 1{\omega(v)}\left(f(v)-\frac 1{(\textrm{deg } v-1)} \sum_{\substack {v'\in V \\ v'\sqsubset\ v}}f(v')\right)
\end{equation*}
Here, deg $v$ is the degree of the vertex $v$. In particular, we will be interested in considering the action of $D$ on the Hilbert space $H$ consisting of weighted $\ell^2$ functions:
\begin{equation*}
H=\ell^2(V, \omega)=\{f:V\to \mathbb C \ :\  \sum_{v\in V} |f(v)|^2\omega(v)<\infty \}\\
\end{equation*}
The Hilbert space $H$ also has a decomposition $H =\bigoplus_{n=0}^\infty \ell^2(V_n, p^{-n})$ coming from the corresponding decomposition of the vertices $V=\bigcup_{n=0}^\infty V_n$. Using the labeling of formula \eqref{vlabel} the operator $D$ becomes,

\begin{equation}\label{formula_for_D}
D f_n(l)=p^{n}\left(f_n(l)-\frac1p\sum_{0\leq j<p}\ f_{n+1}(l+jp^n)\right).
\end{equation}


We can rewrite $D$ using  its Fourier transfrom as,
\begin{equation*}
Df_n(l)=p^n\left(\sum_{0\leq k<p^n}\wh f_n(k)e^{\frac{2\pi ikl}{p^n}}-\frac 1p \sum_{0\leq j<p}\ \sum_{0\leq k<p^{n+1}}\wh f_{n+1}(k)e^{\frac{2\pi ik(l+jp^n)}{p^{n+1}}}\right).
\end{equation*}
Using  the fact that
\begin{equation*}
\sum_{0\leq j<p}e^\frac{2\pi i kj}p=\left\{
\begin{array}{l}
0\ \ \ \textrm{ if }\ p\nmid  k \\ 
p\ \ \ \textrm{ if } \ p\mid k
\end{array}\right.
\end{equation*}
we can simplify the above equation to obtain the following formula for $D$:

\begin{equation}\label{DinFT}
Df_n(l)=p^n \sum_{0\leq k<p^n}\left(\wh f_n(k)-\wh f_{n+1}(pk)\right)e^{\frac{2\pi ikl}{p^n}}.
\end{equation}
\bigskip

In what follows we restrict the operator $D$ to its maximal Hilbert space domain $\mathcal D_{\max}(D)$ given by, 
\begin{equation*}
\mathcal D_{\max}(D)=\{f\in H\ :\ Df\in  H\}. 
\end{equation*}
If we denote $\wh H=\ell^2(\wh V)$, then the map $f\mapsto \wh f$ induces a Hilbert space isomorphism $H \cong \wh H$ by equation  \eqref{parsev}.  We introduce the notation
\begin{equation}\label{Dhat}
\wh D\wh f_n(k)=p^n\left(\wh f_n(k)-\wh f_{n+1}(pk)\right).
\end{equation}
The operator $\wh D$ will also be considered on  its maximal domain, $\mathcal D_{\max}(\wh D)=\{\wh f\in \wh H\ |\ \wh D\wh f\in  \wh H\}$. 
The Fourier transform on the $p$-adic tree establishes a unitary equivalence between $D$ and $\wh D$.
\smallskip

We start the analysis of $D$ with a computation of its kernel. Notice that formula \eqref{DinFT} shows 
\begin{equation}\label{kerD}
{\text Ker} D=\{f=(f_n)\ |\wh f_n(k)= \wh f_{n+1}(pk)\}.
\end{equation}

\begin{prop}
The kernel of the operator $D$ defined on its maximal domain $\mathcal D_{\max}(D)$ in $\mathcal H$ is trivial.
\end{prop}

\begin{proof}
Suppose $f \in \mathcal D_{\textrm{max}}(D)$ such that $Df=0$. From formula \eqref{DinFT} we see that the Fourier coefficients of $f$ satisfy $\wh f_n(k)= \wh f_{n+1}(pk)$. Let us assume that $\wh f_{n_0}(k_0) \neq 0$ for some $n_0, k_0$. Now we estimate the $\ell^2$ norm of $f$ as follows:
\begin{equation*} 
\|f\|^2= \sum_{n\geq 0} \sum_{0 \leq k < p^n}|\wh f_n(k)|^2 \geq \sum_{n\geq n_0} \sum_{0 \leq k < p^n}|\wh f_n(k)|^2 \geq  \sum_{i \geq 0} |\wh f_{n_0+i}(p^ik_0)|^2
\end{equation*}
Since $f \in \ker D$, $\wh f_{n_0}(k_0)=\wh f_{n_0+i}(p^ik_0)$ for each $i$. Therefore the sum $\sum_{i \geq 0} |\wh f_{n_0+i}(p^ik_0)|^2$ is infinite which contradicts the fact that $f$ is in $\ell^2(V, \omega)$. Thus, all Fourier coefficients of $f$ are zero and hence $f \equiv 0$.

\end{proof}

\begin{theo}
The operator $D$ defined on its maximal domain in $H$ is invertible with bounded inverse.\end{theo}

\begin{proof}

Given $g\in H$ we need to solve $Df_n(l)=g_n(l)$ for $f \in \mathcal D_{\textrm{max}}(D)$ in terms of $g$. This can be done easier with the Fourier transforms of $f$ and $g$. Using formula \eqref{DinFT} above for the Fourier transform of $Df_n(l)$ we get

\begin{equation}\label{Fourier_coeff_eqn}
\wh g_n(k)=p^n\Big(\wh f_n(k)-\wh f_{n+1}(pk)\Big).
\end{equation}

This equation has a unique solution in $H$ given by
\begin{equation}\label{dinvft}
\wh f_n(k)= \sum_{i=n}^{\infty}\frac{\wh g_i(l p^{\alpha -n+i})}{p^i}=:\wh D^{-1}\wh g_n(k),
\end{equation}
where $k$ is written in the form $k=lp^{\alpha}$ with $p \nmid l$. The limits in the above sum are determined by requiring that $\wh f_n(k)\to 0$ as $n\to\infty$. Using the Cauchy-Schwartz inequality we can estimate the  pointwise norm of $\wh D^{-1}$:
\begin{equation}\label{absolute_value_of _fn}
|\wh D^{-1}\wh g_n(k)|\leq \left(\sum_{i=n}^{\infty}\wh g_i(lp^{\alpha-n+i})^2\right)^{1/2} \cdot \left(\sum_{i=n}^{\infty} \frac1{p^{2i}}\right)^{1/2}= \frac{p^{-n}}{\sqrt{1-p^{-2}}} \cdot \|\wh g\|.
\end{equation}
Hence the formula for $\wh f_n(k)$ is well defined. Thus, we can now write 
\begin{equation*}
D^{-1}g_n(l)= \sum_{n=0}^{\infty}\sum_{0 \leq l < p^n} \wh f_n(k)e^{\frac{2\pi ikl}{p^n}}=\sum_{n=0}^{\infty}\sum_{0 \leq l < p^n} \sum_{i=n}^{\infty}\frac{\wh g_i(l p^{\alpha -n+i})}{p^i}\; e^{\frac{2\pi ikl}{p^n}}.
\end{equation*}
Below, we will prove that $D^{-1}Df=f$. The proof of $DD^{-1}g=g$  can be done in a similar way.

In fact, we will work with the Fourier transform and show that $\wh D^{-1}\wh D \wh f=\wh f$. Using formula \eqref{dinvft} we write
\begin{equation*}
\wh D^{-1}\wh D \wh f_n(k)= \sum_{i=n}^{\infty} \frac {\wh D \wh f_i(lp^{\alpha -n+i})}{p^i}=\sum_{i=n}^{\infty}\wh f_i(lp^{\alpha -n+i})-\sum_{i=n}^{\infty}\wh f_{i+1}(lp^{\alpha -n+i+1})
\end{equation*} 
where $k=lp^{\alpha }$. Since the above sum is telescopic the only term that survives is $\wh f_n(lp^{\alpha})= \wh f_n(k)$. Hence the claim is proved.
\smallskip

Estimating the norm of $D^{-1} $ using formula \eqref{absolute_value_of _fn} we get:
\begin{equation*}
\begin{aligned}
\|D^{-1} g\|^2 &= \sum_{n=0}^{\infty}\sum_{0 \leq k < p^n}|\wh D^{-1}\wh g_n(k)|^2 \leq \sum_{n=0}^{\infty}\sum_{0 \leq k < p^n} \left(\frac{p^{-2n}}{1-p^{-2}}\right)\|\wh g\|^2\\
&= \frac 1{(1-p^{-2})(1-p^{-1})}\|\wh g\|^2
\end{aligned}
\end{equation*}
showing that the inverse is bounded.

\end{proof}
\bigskip

We end this section with a discussion of the unique continuation property of the operator $D$ on the space $\mathcal E^*(V)$ of all functions on the vertices of the $p$-adic tree.  In order to discuss the theorem we need a concept of a limit at the boundary of the tree. To define this limit first we make the following observation.
If $\phi \in \mathcal E(\mathbb Z_p)$ then we can find a positive integer $n_{\phi}$ and a function $\tilde{\phi}$ such that,
\begin{enumerate}[i.]
\item $\tilde{\phi} \in \ell^2(V_n, \omega)$ for all $n\geq n_{\phi}$
\item $\tilde{\phi}(v)=\phi(x)$ for every $x\in v$.
\end{enumerate} 
\smallskip

Notice that, for the integral of $\phi \in \mathcal E(\mathbb Z_p)$ we have the formula,
\begin{equation*}
\int\phi(x)\,d_px=\sum_{v\in V_n} \tilde{\phi}(v)\,p^{-n}
\end{equation*}
valid for every $n\geq n_\phi$.
\smallskip

\begin{defin}
Let $f=\{f_n\}_{n=0}^\infty\in \mathcal E^*(V)$ be a function on $V$. Then $f$ is said to converge weakly to a distribution $T\in \mathcal E^*(\mathbb Z_p)$ on the boundary of the tree $\{V,E\}$, if for every $\phi \in \mathcal E(\mathbb Z_p)$,
\begin{equation*}
\lim_{n\to\infty}\sum_{v\in V_n} f_n(v)\tilde{\phi} (v)\,p^{-n}:=T(\phi)
\end{equation*}
exists. In this case, we say that $f_n$ has a limit at the boundary, namely $T$.
\end{defin}

The following key lemma provides a necessary and sufficient condition for the weak convergence of a sequence of functions living on the $p$-adic tree. 
\begin{lem}\label{key_lemma1}
Let $f=\{f_n(k)\ :\ n\in\mathbb N,\  0\leq k<p^n\}$ be a complex-valued function on the vertices of the $p$-adic tree. Then $f$ has a limit at the boundary if and only if \ $\lim_{n\to\infty}\wh f_n (lp^{n-m})$ exists for every $p\nmid l$ and $m\leq n$.
\end{lem}

\begin{proof}
Suppose $f_n$ has a limit at the boundary. Using formula \eqref{formula_for_fn} we see that
\begin{equation*}
\wh f_n(lp^{n-m})=\frac{1}{p^n}\sum_{k=0}^{p^n-1}f_n(k)e^{-2\pi i\frac{kl}{p^m}}.
\end{equation*}
Notice that the character $\chi_{n,k}$ in formula \eqref{nkchar} is also a test function and so the expression $e^{-2\pi i\frac{kl}{p^m}}$ appearing on the right hand side of the above equation is equal to $\tilde\chi_{m,l}(n,k)$. By our assumption, the limit of the right hand side exists as $n\to \infty$ for any $p\nmid l$ and $m\leq n$. So the limit of the left hand side also exists as $n \rightarrow \infty$.

Now suppose that $\lim_{n\to \infty} \wh f_n(lp^{n-m})$ exists for all $p \nmid l$ and $m\leq n$. Every test function is a finite linear combination of characters because, test functions have finite Fourier expansions. Consequently, by linearity, it is sufficient to check that the limit
\begin{equation*}
\lim_{n\to\infty}\sum_{v\in V_n} f_n(v)\tilde\chi_{m,l}(v)\,p^{-n}
\end{equation*} 
exists.
But the existence of the above limit is guaranteed by the hypothesis since
\begin{equation*}
\sum_{v\in V_n} f_n(v)\tilde\chi_{m,l}(v)\,p^{-n}=\frac{1}{p^n}\sum_{k=0}^{p^n-1}f_n(k)e^{-2\pi i\frac{kl}{p^m}}=\wh f_n(lp^{n-m}).
\end{equation*}

\end{proof}
\bigskip

The following theorem describes the unique continuation property of the operator $D$. 
\begin{theo}\label{unique_cont_prop}
Let $f\in$ Ker $D$. Then the following are true.
\begin{enumerate}[(i).]
\item $f$ has a limit at the boundary.
\item If the limit of $f$ at the boundary is equal to zero then $f$ is identically zero.
\item If $T\in \mathcal E^*(\mathbb Z_p)$ then there exists $g\in$ Ker $D$ such that the limit of $g$ at the boundary is equal to $T$.
\end{enumerate}
\end{theo}

\begin{proof}

To prove part (i) of the theorem we observe from formula \eqref{DinFT} that
the operator $D$ on $\mathcal E^*(V)$ has an infinite dimensional kernel. 
Using the previous lemma it suffices to show that if $f\in$ Ker $D$ then the limit $\lim\limits_{n\to\infty}\wh f_n (lp^{n-m})$ exists for every $p\nmid l$ and $m$. 
If $p\nmid l$ and $n \geq m$ then
\begin{equation*}
\begin{aligned}
\wh f_n (lp^{n-m})&=\wh f_{n-1} (lp^{n-m-1})=\wh f_{n-2} (lp^{n-m-2})=\ldots\\
&=\wh f_m(l)
\end{aligned}
\end{equation*}
So $\lim_{n\to\infty}\wh f_n (lp^{n-m})$ exists and is equal to $\wh f_m(l)$. 
\smallskip

To prove the second part of the theorem, we first establish another formula for $f_n(k)$ in terms of its Fourier coefficients.

Starting with formula \eqref{FIof_f}, we decompose the sum on the right hand side of \eqref{FIof_f} such that one sum is over all $l$ divisible by $p$ and the other over all $l$ such that $p \nmid l$ to obtain
\begin{equation*}
f_n(k)=\sum_{\substack{p\ \nmid\ l \\0\leq l< p^n}} \wh f_n(l)e^{\frac{2\pi ikl}{p^n}} + \sum_{0\leq l< p^{n-1}}\wh f_n(lp)e^{\frac{2\pi ikl}{p^{n-1}}}
\end{equation*}
Iterating this we get,
\begin{equation}\label{expression_for_f}
\begin{aligned}
f_n(k)&=\sum_{\substack{p\ \nmid\ l \\0\leq l< p^n}} \wh f_n(l)e^{\frac{2\pi ikl}{p^n}} + \sum_{\substack{p\ \nmid\ l \\0\leq l< p^{n-1}}}\wh f_n(lp)e^{\frac{2\pi ikl}{p^{n-1}}} + \cdots +\sum_{0\leq l< p}\wh f_n(lp^{n-1})e^{\frac{2\pi ikl}{p}}\\
&= \sum_{m=1}^{n}\sum_{\substack{(l,p)=1 \\ 0\leq l<p^m}} \wh f_n (lp^{n-m})e^\frac{2\pi ikl}{p^m}.
\end{aligned}
\end{equation}

 If the limit of $f$ at the boundary is equal to zero then 
\begin{equation*}
\lim_{n\to \infty}\frac1{p^n}\sum_{0 \leq k<p^n}f_n(k)e^{-\frac{2\pi ikl}{p^m}}=\lim_{n\rightarrow \infty} \wh f_n(lp^{n-m})=0
\end{equation*}
which in turn implies from part (i) that $\wh f_m(l)=0$ for each $m$ and $0\leq l< p^m$ such that  $p\nmid l $. Using formula \eqref{expression_for_f} above we see that 
\begin{equation*}
f_n(k)=\sum_{m=1}^{n}\sum_{\substack{(l,p)=1 \\ 0\leq l<p^m}} \wh f_n (lp^{n-m})e^\frac{2\pi ikl}{p^m}=\sum_{m=1}^{n}\sum_{\substack{(l,p)=1 \\ 0\leq l<p^m}} \wh f_m (l)e^\frac{2\pi ikl}{p^m}.
\end{equation*}
Therefore $f_n(k)\equiv 0$ for each $n$.
\smallskip

Finally, we prove the third part of the statement. Any distribution $T\in \mathcal E^*(\mathbb Z_p)$ is completely determined by its Fourier coefficients  $T_{n,k}:=T(\overline{\chi_{n,k}})$ which, moreover, can be arbitrary numbers. For given $T_{m,l}$ we would like to find $g\in $ Ker $D$ such that 

\begin{equation*}
\begin{aligned}
\lim_{n \rightarrow \infty}\sum_{v\in V_n} g_n(v)\tilde\chi_{m,l}(v)\,p^{-n}&=\lim_{n \rightarrow \infty}p^{-n} \sum_{0 \leq k < p^n} g_n(k)e^{\frac{-2\pi ikl}{p^m}}=\lim_{n\rightarrow \infty} \wh g_n(lp^{n-m})=\\
&=T_{m,l}.\\
\end{aligned}
\end{equation*}
If $p\nmid l$ then from part (i) we see that the Fourier coefficients of $g$ are given by $\wh g_m(l)=T_{m,l}$. If $p \mid l$ let $l=p^rk$ with $\gcd(p,k)=1$. Then: 
\begin{equation*}
\wh{g_n}(lp^{n-m})=\wh g_{m-r}(k)=T_{m-r,l}. 
\end{equation*}
Thus $g\in$ Ker $D$ is completely determined by the conditions,
\begin{equation*}
\wh g_m(l)=\left\{
\begin{array}{l}
T_{m,l}\ \ \ \ \ \,\textrm{ if }p\mid l \\ 
T_{m-r,l}\ \ \ \textrm{ if }l=p^r k \textrm{ with } \gcd(p,k)=1.
\end{array}\right.
\end{equation*}
\end{proof}

\newpage

\section{The spectral triple}\label{sec4}

The goal of this section is to construct a spectral triple for the C$^*$- algebra $A=C(\mathbb Z_p)$, the space of continuous functions on $\mathbb Z_p$.  For this construction we use the Hilbert space $H$ and the (non self-adjoint) operator $D$ that was introduced in the previous section. 
First, we construct a new Hilbert space $\mathcal H=H \bigoplus  H$ and a new operator $\mathcal D$  on $\mathcal H$ given by:
\begin{equation*}
\mathcal D= \left(
\begin{array}{cc}
0 & D \\
D^{\ast} & 0
\end{array}\right).
\end{equation*}
The Hilbert space $\mathcal H$ has a natural $\mathbb Z/2\mathbb Z$ grading with respect to which the operator $\mathcal D$ is odd (off diagonal).

We let  $\mathcal A$ be the space of all Lipschitz functions on $\mathbb Z_p$ which is a dense $*$-subalgebra of $A$. Recall that a function $\phi:\mathbb Z_p \to \mathbb C$ is called Lipschitz ( or Lipschitz continuous) if there is a constant $c>0$ such that $|\phi(x)-\phi(y)|\leq c\,\rho_p(x,y)$ for all $x,y \in \mathbb Z_p$. The smallest such constant, called the Lipschitz norm of $\phi$, will be denoted by $L(\phi )$:
\begin{equation*}
L(\phi ) = \sup\limits_{x\ne y} \frac{|\phi (x) - \phi (y)|}{\rho_p(x, y)}.
\end{equation*}

The algebra $A$ has a natural representation $\pi$ in $H$ given by $\pi(\phi)f_n(k)=\phi(k)f_n(k)$. 
Define $\Pi: A \to \mathcal B(\mathcal H)$  to be  $\Pi=\pi\oplus\pi$, i.e.

\begin{equation*}
\Pi(\phi)\left(\begin{array}{c}f_n(k)\\g_n(k)\end{array}\right)=\left(\begin{array}{c}\phi(k)f_n(k)\\ \phi(k)g_n(k)\end{array}\right). 
\end{equation*}
The representation $\Pi$ is even (diagonal) with respect to the grading on $\mathcal H$. Additionally we make the following simple observation.

\begin{prop}\label{ndprop}
The representation $\Pi: A \to \mathcal B(\mathcal H)$ is faithful and non-degenerate.
\end{prop}
\begin{proof}
Notice that 
\begin{equation*}
||\Pi(\phi)||=||\pi(\phi)||=\sup_{k=0,1,2\ldots}|\phi(k)|,
\end{equation*}
since $\pi(\phi)$ is a diagonal operator in $H$. But the set of nonnegative integers is dense in $\mathbb Z_p$, and $\phi$ is continuos, so $||\Pi(\phi)||=\sup_{x\in\mathbb Z_p}|\phi(x)|$, and so the representation of  $A=C(\mathbb Z_p)$ is faithful. Since $\Pi(1)=I$, the representation is also non-degenerate.

\end{proof}

The following theorem is our main result in this section.

\begin{theo}
The triple $(\mathcal A, \mathcal H,\mathcal D)$ defined above is an even spectral triple.
\end{theo}

\begin{proof}

If we can show that $D^{-1}$ is a Hilbert-Schmidt operator, then the compactness of $(1+\mathcal D^2)^{-1/2}$ easily follows from it by using functional calculus.
To prove that $D^{-1}$ is Hilbert-Schmidt we estimate its Hilbert-Schmidt norm as follows.
\smallskip

Using formula \eqref{dinvft} we write $D^{-1}$ in Fourier transform as,
\begin{equation*}
\wh D^{-1}\wh g_n(k)=\sum_{i=0}^{\infty} \sum_{0 \leq r < p^i}K(n,k,i,r) \wh g_i(r).
\end{equation*}
The integral kernel $K$ in the above expression is given by
\begin{equation*}
\begin{aligned}K(n,k,i,r)=\begin{cases}
\frac 1{p^i} & \textrm{  if } i \geq n \textrm{ and } r=lp^{\alpha -n+i} \\ 
0 & \textrm{  otherwise }
\end{cases}\\
\end{aligned}
\end{equation*}
where $\alpha$ is the integer such that $k=lp^{\alpha}$ with $p \nmid l$.
From this formula we can now compute the Hilbert-Schmidt norm of $D^{-1}$ as follows.
\begin{equation*}
\begin{aligned}
\|D^{-1}\|_{HS}^2&=\sum_{i=0}^{\infty}\sum_{0 \leq r < p^i} \sum_{n=0}^{\infty} \sum_{0 \leq k < p^n}|K(n,k,i,r)|^2= \sum_{n=0}^{\infty} \sum_{0 \leq k < p^n}\sum_{i=n}^{\infty} |K(n,k,i,lp^{\alpha -n +i})|^2=\\
&= \sum_{n=0}^{\infty} \sum_{0 \leq k < p^n}\sum_{i=n}^{\infty} \frac 1{p^{2i}} =\frac{1}{(1-p^{-1})(1-p^{-2})}<\infty.
\end{aligned}
\end{equation*}
Thus, $D^{-1}$ is indeed a Hilbert-Schmidt operator.
\smallskip

We would now like to prove that the commutator $[\mathcal D, \Pi(\phi)]$ is bounded for any $\phi\in\mathcal A$.
Notice that:

\begin{equation}\label{Dcomm}
[\mathcal D, \Pi(\phi)] \left(
\begin{array}{c}
f_n(k)\\g_n(k)
\end{array}\right)=\left(
\begin{array}{c}
D\pi(\phi) g_n(k)-\pi(\phi ) Dg_n(k)\\ D^{\ast}\pi(\phi) f_n(k)-\pi(\phi)  D^{\ast} f_n(k).
\end{array}\right)
\end{equation}
The first component of the above expression can be written as
\begin{equation}\label{comform}
\begin{aligned}
D\pi(\phi) g_n(k)-\pi(\phi ) Dg_n(k)&=p^n \left[ \phi(k)g_n(k)-\frac 1p \sum_{i=0}^{p-1} \phi(k+ip^n) g_{n+1}\big( k+ip^n\big)\right]\\
&-p^n\left[ \phi(k)g_n(k)-\frac 1p \sum_{i=0}^{p-1} \phi(k)g_{n+1}\big( k+ip^n\big)\right]\\
&=p^{n-1}\sum_{i=0}^{p-1}\big[\phi(k)-\phi(k+ip^n)\big]g_{n+1}( k+ip^n).
\end{aligned}
\end{equation}
Thus, using the Lipschitz condition, we can estimate the pointwise norm of the commutator as follows.
\begin{equation*}
\begin{aligned}
\big| D\pi(\phi) g_n(k)-\pi(\phi ) Dg_n(k)\big|&\leq p^{n-1}\sum_{i=0}^{p-1}\big|\phi(k)-\phi(k+ip^n)\big||g_{n+1}( k+ip^n)|\\
& =\frac {L(\phi)}p \sum_{i=0}^{p-1}|g_{n+1}( k+ip^n)|
\end{aligned}
\end{equation*}
 
Consequently, we have:
\begin{equation*}
\begin{aligned}
\| D \pi(\phi) g-\pi(\phi ) D g\|^2 &=\sum_{n=0}^{\infty}\ \sum_{0\leq k < p^n}p^{-n}\big| D\pi(\phi) g_n(k)-\pi(\phi ) Dg_n(k)\big|^2\\
&\leq \sum_{n=0}^{\infty}\ \sum_{0\leq k < p^n}p^{-n-2}L(\phi)^2{\left(\sum_{i=0}^{p-1}|g_{n+1}( k+ip^n)|\right)}^2.
\end{aligned}
\end{equation*}
Using the fact that $\left(\sum_{k=1}^n a_k\right)^2\leq n \sum_{k=1}^n a_k^2$ for any $n$ number of positive numbers $a_k$  we see that,
\begin{equation*}
\begin{aligned}
\| D \pi(\phi) g-\pi(\phi ) D g\|^2 & \leq \sum_{n=0}^{\infty}\ \sum_{0\leq k < p^n} \ \sum_{0 \leq i \leq p-1} p^{-n-1} L(\phi)^2 \left|g_{n+1}( k+ip^n)\right|^2\\
&\leq \sum_{n=0}^{\infty}\ \sum_{0\leq l < p^{n+1}} p^{-n-1} L(\phi)^2 \left|g_{n+1}( l)\right|^2\\
& = L(\phi)^2 \|g\|^2.
\end{aligned}
\end{equation*}
In other words $\| D \pi(\phi) -\pi(\phi ) D \|\leq L(\phi)$. 
By taking the adjoint in the above formula we also get the norm estimate $\| D^{\ast}\pi(\phi) -\pi(\phi ) D^{\ast}\|=\| D \pi(\phi) -\pi(\phi ) D \|\leq L(\phi)$.
Consequently the commutator $[\mathcal D, \Pi(\phi)]$ is bounded for any $\phi\in\mathcal A$. This concludes that $(\mathcal A, \mathcal H,\mathcal D)$ is an even spectral triple.

\end{proof}

\section{The distance}

In this last part of the paper we study metric properties of the spectral triple constructed in the previous section. We will consider the Lipschitz seminorm as a function (possibly infinite valued) on the whole C$^*$-algebra, $A=C(\mathbb Z_p)$. i.e. $L:A\to [0,\infty]$.
The spectral seminorm, $L_{\mathcal D}(\phi)$,  of $\phi\in A$, is defined by $L_{\mathcal D}(\phi)=\|[\mathcal D, \Pi(\phi)]\|$, where unbounded operators are considered to have infinite norm. It is clear from formula \eqref{Dcomm} that $L_{\mathcal D}(\phi)=\|[D, \pi(\phi)]\|$.

The seminorms $L$ and $L_{\mathcal D}$ define two metrics dist and $\textrm{dist}_{\mathcal D}$ respectively, on the space of $p$-adic integers $\mathbb Z_p$ via the Connes distance formula:
 
\begin{equation*}
\begin{aligned}
\textrm{dist}(x,y)&=\sup_{\phi\in A} \{|\phi(x)-\phi(y)|: L(\phi)\leq 1\}, \,\, x,y \in \mathbb Z_p,\;\;\;\; \textrm{ and}\\
\textrm{dist}_{\mathcal D}(x,y)&=\sup_{\phi\in A} \{|\phi(x)-\phi(y)|: L_{\mathcal D}(\phi)\leq 1\}, \,\, x,y \in \mathbb Z_p.
\end{aligned}
\end{equation*}

The goal of this section is to compare these two metrics. First we observe the following general fact.

\begin{prop}
The metric, $\textrm{dist}(x,y)$, induced by the Lipschitz seminorm is equal to the usual p-adic metric.
\end{prop}

\begin{proof}
For every $\phi\in A$ with $L(\phi)\leq 1$ we have $|\phi(x)-\phi(y)|\leq |x-y|_p$. Consequently, $\textrm{dist}(x,y)\leq |x-y|_p$.
To prove the equality we consider the function $\phi_x(z)=|z-x|_p$ which is Lipschitz continuous with $L(\phi_x)=1$. Moreover, for this function we have
$|\phi_x(x)-\phi_x(y)|= |x-y|_p$, and the result follows.
\end{proof}

The results in this section are a consequence of the following two lemmas describing the spectral seminorm $L_{\mathcal D}$. In the first lemma we give a concrete formula for the norm $\|[D, \pi(\phi)]\|$, and then use it to establish inequalities showing equivalence of seminorms $L$ and $L_{\mathcal D}$ in the second lemma.

\begin{lem}\label{LDlemma}
For $\phi\in \mathcal A$ we have:
\begin{equation*}
L_{\mathcal D}(\phi)=\|[D, \pi(\phi)]\|=\left(\sup_{\substack {n \\ 0\leq k < p^n}}\frac1p\sum_{i=1}^{p-1}\frac{|\phi(k)-\phi(k+i p^n)|^2}{|k-(k+ip^n)|_p^2}\right)^{1/2}.
\end{equation*}
\end{lem} 

\begin{proof}
Notice that, because $\phi$ is a Lipschitz function, the ratio $\sum_{i=1}^{p-1}\frac{|\phi(k)-\phi(k+ip^n)|^2}{|k-(k+ip^n)|_p^2}$ is bounded by a uniform constant, hence the sup above exists. Using formula \eqref{comform} we see that:
\begin{equation*}
\begin{aligned}
\left| D \pi(\phi) g_n(k)-\pi(\phi)  D g_n(k)\right|&= \left|p^{n-1}\sum_{i=0}^{p-1}[\phi(k)-\phi(k+ip^n)]g_{n+1}( k+ip^n)\right|\\
&\leq\frac1p \sum_{i=0}^{p-1}\frac{|\phi(k)-\phi(k+ip^n)|}{p^{-n}}|g_{n+1}( k+ip^n)|.
\end{aligned}
\end{equation*}
Thus, 
\begin{equation*}
\| D \pi(\phi) g- \pi(\phi)  D g\|^2 \leq \sum_{n=0}^{\infty}\sum_{0\leq k < p^n}p^{-n}\frac1{p^2}\left(\sum_{i=0}^{p-1}\frac{|\phi(k)-\phi(k+ip^n)|}{p^{-n}}|g_{n+1}( k+ip^n)|\right)^2.
\end{equation*}
Using Cauchy-Schwartz inequality we can estimate the norm of the commutator as follows.
\begin{equation*}
\begin{aligned}
&\| D \pi(\phi) g- \pi(\phi)  D g\|^2 \leq \sum_{n=0}^{\infty}\sum_{0\leq k < p^n}\frac{p^{-n}}{p^2}\left(\sum_{i=0}^{p-1}\frac{|\phi(k)-\phi(k+ip^n)|^2}{p^{-2n}}\right)\left(\sum_{j=0}^{p-1}|g_{n+1}( k+jp^n)|^2\right)\\
 &\leq \sup_{n\geq 0}\sup_{0\leq k < p^n}\frac1p\sum_{i=0}^{p-1}\frac{|\phi(k)-\phi(k+ip^n)|^2}{p^{-2n}}\left(\sum_{n=0}^{\infty}\sum_{0\leq k < p^n}\sum_{j=0}^{p-1}p^{-n-1}|g_{n+1}( k+jp^n)|^2\right)\\
 &=\sup_{n \geq 0}\sup_{0\leq k < p^n}\frac1p\sum_{i=1}^{p-1}\frac{|\phi(k)-\phi(k+ip^n)|^2}{|k-(k+ip^n)|_p^2} \|g\|^2.
\end{aligned}
\end{equation*}
So we have shown that,
\begin{equation*}
\| D \pi(\phi) - \pi(\phi)  D \|^2 \leq \sup_{\substack {n \\ 0\leq k < p^n}}\frac1p\sum_{i=1}^{p-1}\frac{|\phi(k)-\phi(k+ip^n)|^2}{|k-(k+ip^n)|_p^2}.
\end{equation*}

To prove the equality of both sides we will make a particular choice of $g$. For any $\epsilon >0$ there are numbers $n_{\epsilon}, \,k_{\epsilon}$ such that 
\begin{equation*}
0 \leq \sup_{n \geq 0}\sup_{0\leq k < p^n}\left(\frac1p\sum_{i=1}^{p-1}\frac{|\phi(k)-\phi(k+ip^n)|^2}{|k-(k+ip^n)|_p^2}\right)-\frac1p\sum_{i=1}^{p-1}\frac{|\phi(k_{\epsilon})-\phi(k_{\epsilon}+i p^{n_{\epsilon}})|^2}{|k_{\epsilon}-(k_{\epsilon}+ip^{n_{\epsilon}})|_p^2} < \epsilon
\end{equation*}
Given $n_{\epsilon}, \,k_{\epsilon}$ we define the function $g^{\epsilon}$ by:
\begin{equation*}
g^{\epsilon}_{n+1}(k+ip^{n})=\begin{cases}
\frac{\overline{\phi(k_{\epsilon})}-\overline{\phi(k_{\epsilon}+ip^{n_{\epsilon}})}}{p^{-n_{\epsilon}}} & \textrm {; if } (n,k)=(n_{\epsilon},k_{\epsilon})\\
0 & \textrm {; otherwise}
\end{cases}
\end{equation*}
where $i=0,1,\ldots,p-1$.
Then we have
\begin{equation*}
\begin{aligned}
\| D \pi(\phi) g^{\epsilon}- \pi(\phi) D g^{\epsilon}\|^2&=\sum_{n=0}^{\infty}\sum_{0\leq k < p^n}p^{-n}\frac1{p^2}\left|\sum_{i=1}^{p-1}\frac{\big(\phi(k)-\phi(k+ip^n)}{|k-(k+ip^n)\big)|_p}g^{\epsilon}_{n+1}(k+ip^n)\right|^2\\
&=\frac{p^{-n_{\epsilon}}}{p^2}\left|\sum_{i=1}^{p-1}\frac{\big(\phi(k_{\epsilon})-\phi(k_{\epsilon}+i p^{n_{\epsilon}})\big)}{p^{-n_{\epsilon}}}\cdot \frac{\big(\overline{\phi(k_{\epsilon})}-\overline{\phi(k_{\epsilon}+ip^{n_{\epsilon}})}\big)}{p^{-n_{\epsilon}}}\right|^2\\
&=\frac{p^{-n_{\epsilon}}}{p^2}\left(\sum_{i=1}^{p-1}\frac{|\phi(k_{\epsilon})-\phi(k_{\epsilon}+i p^{n_{\epsilon}})|^2}{p^{-2n_{\epsilon}}}\right)^2\\
\end{aligned}
\end{equation*}

The above expression can be rewritten as
\begin{equation*}
\begin{aligned}
\| D \pi(\phi) g^{\epsilon}- \pi(\phi) D g^{\epsilon}\|^2&=\frac{p^{-n_{\epsilon}}}{p^2}\left(\sum_{i=1}^{p-1}\frac{|\phi(k_{\epsilon})-\phi(k_{\epsilon}+i p^{n_{\epsilon}})|}{p^{-n_{\epsilon}}}\cdot |g^{\epsilon}_{n_{\epsilon}+1}(k_{\epsilon}+ip^{n_{\epsilon}})|\right)^2\\
&=\frac1p\left(\sum_{i=1}^{p-1}\frac{|\phi(k_{\epsilon})-\phi(k_{\epsilon}+i p^{n_{\epsilon}})|^2}{p^{-2n_{\epsilon}}}\right)\|g^{\epsilon}\|^2.\\
\end{aligned}
\end{equation*}
This computation shows that for every $\epsilon$,
\begin{equation*}
\begin{aligned}
\|[D, \pi(\phi)]\|^2 &\geq  \frac1p\left(\sum_{i=1}^{p-1}\frac{|\phi(k_{\epsilon})-\phi(k_{\epsilon}+i p^{n_{\epsilon}})|^2}{p^{-2n_{\epsilon}}}\right)\\
 &\geq\sup_{n \geq 0}\sup_{0\leq k < p^n}\frac1p \sum_{i=1}^{p-1} \left| \frac{\phi(k)-\phi(k+i p^{n})}{p^{-n}}\right|^2- \frac{\epsilon}{p}.
\end{aligned}
\end{equation*}
Since this is true for every $\epsilon$, we let $\epsilon \rightarrow 0$ to obtain
\begin{equation*}
\|[D, \pi(\phi)]\|^2 \geq \sup_{n \geq 0}\sup_{0\leq k < p^n}\frac1p \sum_{i=1}^{p-1} \left| \frac{\phi(k)-\phi(k+i p^{n})}{p^{-n}}\right|^2.
\end{equation*}
Since we already established the inequality the other way around, we have proved the lemma.
\end{proof}

Next we are going to use Lemma \ref{LDlemma} to find estimates for $L_{\mathcal D}$ in terms of $L$, which in essence says that the two seminorms are equivalent.
\smallskip

\begin{lem}\label{LLDlemma}
With $L$ and $L_{\mathcal D}$ defined as above, for every $\phi\in\mathcal A$ we have:
\begin{equation*}
\sqrt{\frac {p-1}{p}}\,L(\phi)\geq L_{\mathcal D}(\phi) \geq \frac{(p-1)}{2p\sqrt p}\,L(\phi).
\end{equation*}
\end{lem} 
\begin{proof}
Let $\phi\in\mathcal A$ be a Lipschitz function. The inequality $L_{\mathcal D}(\phi) \leq\sqrt{\frac {p-1}{p}}L(\phi)$ can simply be obtained by  estimating the formula for $L_{\mathcal D}(\phi)$ from Lemma \ref{LDlemma}.  Using the Lipschitz property of $\phi$ we obtain,
\begin{equation*}
\frac1p\sum_{i=1}^{p-1}\frac{|\phi(k)-\phi(k+i p^n)|^2}{|k-(k+ip^n)|_p^2}\leq \frac1p\sum_{i=1}^{p-1}L(\phi)^2=\frac {p-1}{p}L(\phi)^2.
\end{equation*}
By taking the supremum of both sides we see that $L_{\mathcal D}(\phi) \leq\sqrt{\frac {p-1}{p}}L(\phi)$.
\smallskip

To prove the other side of the inequality we notice that for any $\phi\in\mathcal A$, all $n$ and $0\leq k<p^n$ we have:
\begin{equation*}
\frac1p \sum_{i=1}^{p-1}\left|\frac{\phi(k)-\phi(k+i p^{n})}{p^{-n}}\right|^2\leq L_{\mathcal D}(\phi)^2.
\end{equation*}
Since the above sum consists of positive terms each term can be estimated as 
\begin{equation}\label{ldstep1}
\frac{\left|\phi(k)-\phi(k+i p^{n})\right|}{p^{-n}} \leq \sqrt p\, L_{\mathcal D}(\phi).
\end{equation}

For any $\phi\in\mathcal A$, and $x\ne y$ we want to estimate the quotient $\frac{|\phi (x) - \phi (y)|}{|x-y|_p}$ in terms of $L_{\mathcal D}(\phi)$.
If $|x-y|_p=p^{-n}$ for some $n\geq 0$, then $x, y$ are in one of the $(n,k)$ balls of formula \eqref{vlabel} which would  mean that there is a nonnegative integer $k<p^n$ such that $x, y$ have $p$-adic representations of the form:
\begin{equation*}
\begin{aligned}
x&=k+a_0p^n+a_1p^{n+1}+\ldots+a_ip^{n+i}+\ldots\\
y&=k+b_0p^n+b_1p^{n+1}+\ldots+b_jp^{n+j}+\ldots ,
\end{aligned}
\end{equation*}   
where $a_0\ne b_0$.
Using telescoping technique we first estimate the partial difference:
\begin{equation*}
\begin{aligned}
&|\phi(k+a_0p^n+a_1p^{n+1}+\ldots+a_Np^{n+N})-\phi(k)|\\
& \leq |\phi(k)-\phi(k+a_0p^n)|\,+\,|\phi(k+a_0p^n)-\phi(k+a_0p^n+a_1p^{n+1})|\,+ \ldots + \\
&|\phi(k+a_0p^n+ \ldots + a_{N-1}p^{n+N-1})-\phi(k+a_0p^n+\ldots + a_{N}p^{n+N})|\\
&\leq  \sqrt p\, L_{\mathcal D}(\phi)(p^{-n}+p^{-(n+1)}+p^{-(n+2)}+\ldots+p^{-(n+N)} )\\
&= \sqrt p\, L_{\mathcal D}(\phi)\,\frac{p^{-n}(1-p^{-N-1})}{(1-p^{-1})}.
\end{aligned}
\end{equation*} 
Taking the limit as $N\to\infty$, and using the continuity of $\phi$ we obtain:
\begin{equation*}
|\phi(x)-\phi(k)|\leq \sqrt p\, L_{\mathcal D}(\phi)\,\frac{p^{-n}}{(1-p^{-1})}.
\end{equation*}

Similarly we can show that 
\begin{equation*}
|\phi(k)-\phi(y)|\leq \sqrt p\, L_{\mathcal D}(\phi)\,\frac{p^{-n}}{(1-p^{-1})}.
\end{equation*}
Therefore, using the triangle inequality
\begin{equation*}
|\phi(x)-\phi(y)|\leq  \sqrt p\, L_{\mathcal D}(\phi)\,\frac{2p^{-n}}{(1-p^{-1})}=\frac{2\sqrt p\, L_{\mathcal D}(\phi)}{(1-p^{-1})}\rho_p(x,y).
\end{equation*}
Since the above inequality is true for any $x$ and $y$, the Lipschitz constant $L(\phi)$ must satisfy $L(\phi) \leq \frac{2\sqrt p\, L_{\mathcal D}(\phi)}{1-p^{-1}}$.
So the other side of the inequality,  $L_{\mathcal D}(\phi) \geq \frac{(p-1)}{2p\sqrt p}\,L(\phi)$ is also verified.

\end{proof}
\newpage

\begin{cor}\label{commcor}$\ $\

\begin{enumerate}
\item The algebra $C^1(A):= \{\phi \in A;\ \|[\mathcal D,\pi(\phi)]\| <\infty\}$ is the algebra of Lipschitz functions $\mathcal A$.
\item The commutant $A_{\mathcal D}' = \{\phi \in A;\ [\mathcal D,\pi(\phi)] = 0\}$ is trivial. i.e., $A_{\mathcal D}' = \C I$.
\end{enumerate}
\end{cor}

\begin{proof}
It follows from Lemma \ref{LLDlemma} that $L_{\mathcal D}(\phi)<\infty$ iff $L(\phi)<\infty$. Consequently, the operator $[\mathcal D,\pi(\phi)]$ is bounded iff $L(\phi)<\infty$, i.e., $\phi\in\mathcal A$. This proves the first part of the corollary.
 
To verify that the commutant  $A_{\mathcal D}'$ is trivial once again using Lemma \ref{LLDlemma} we observe that $L_{\mathcal D}(\phi)=0$ iff $L(\phi)=0$. But the only functions with zero Lipschitz norm are constant functions.
\end{proof}
\smallskip

The next theorem states the main result of this section, namely that the metric induced by the spectral triple is equivalent to the $p$-adic norm on $\mathbb Z_p$. 

\begin{theo}
For any $x,y\in \mathbb Z_p$ we have 
\begin{equation*}
\sqrt{\frac {p}{p-1}}|x-y|_p\leq \textrm{ dist}_{\mathcal D}(x,y)\leq \frac{2p\sqrt p}{(p-1)}|x-y|_p . 
\end{equation*}
\end{theo}

\begin{proof}
This result is now a simple consequence of Lemma \ref{LLDlemma}. In fact, we can rewrite the formula for $\textrm{ dist}_{\mathcal D}$ as
\begin{equation*}
\textrm{ dist}_{\mathcal D}(x,y)=\sup_{\phi: L_{\mathcal D}(\phi)\ne 0} \frac{|\phi(x)-\phi(y)|}{L_{\mathcal D}(\phi)}.
\end{equation*}
As noted in the proof of Corollary \ref{commcor} we have that $L_{\mathcal D}(\phi)=0$ iff $L(\phi)=0$. Hence, writing
\begin{equation*}
\frac{|\phi(x)-\phi(y)|}{L_{\mathcal D}(\phi)}=\frac{|\phi(x)-\phi(y)|}{L(\phi)}\frac{L(\phi)}{L_{\mathcal D}(\phi)}
\end{equation*}
and using the inequalities of Lemma \ref{LLDlemma}, completes the proof.
\end{proof}
\bigskip

Finally, we verify that the spectral triple $(\mathcal A,\mathcal H,\mathcal D)$ satisfies the conditions of a compact spectral metric space given in Definition \ref{CQMS}.

\begin{theo}
The pair $(A, L_{\mathcal D})$ defined above is a compact spectral metric space.
\end{theo}

\begin{proof}
In the  proof of Proposition \ref{ndprop} we already verified  that the representation $\Pi$ is non-degenerate. The triviality of the commutant $A_{\mathcal D}'$ was established in Corollary \ref{commcor}.

To verify that the image of the Lipshitz ball $B_{\mathcal D} = \{\phi \in A :\  L_{\mathcal D}(\phi)\leq 1\}$ is precompact in $A/A'_{\mathcal D}$ we first observe that we have a natural identification of $A/A'_{\mathcal D}=C(\mathbb Z_p)/\mathbb C I$ with $A_{\{0\}}:=\{\phi\in C(\mathbb Z_p)\ :\ \phi(0)=0\}$. Let $\phi_n\in A_{\{0\}}$, $n=1,2,\dots$, be a sequence of functions in $A_{\{0\}}$ such that $L_{\mathcal D}(\phi_n)\leq 1$. To prove that  $\{\phi_n\}$ has a convergent subsequence we use Lemma \ref{LLDlemma} and the Ascoli-Arzela theorem.

It follows from Lemma \ref{LLDlemma} that:
\begin{equation*}
 L(\phi_n)\leq\frac{2p\sqrt p}{(p-1)}\, L_{\mathcal D}(\phi)\leq \frac{2p\sqrt p}{(p-1)}.
\end{equation*}
Consequently, for every $n$ and for every $x, y\in\mathbb Z_p$, we have:
\begin{equation*}
|\phi_n(x)-\phi_n(y)|\leq \frac{2p\sqrt p}{(p-1)}\,\rho_p(x,y),
\end{equation*}
which implies that the family $\{\phi_n\}$ is equicontinuous. It is also uniformly bounded because
\begin{equation*}
|\phi(x)|=|\phi_n(x)-\phi_n(0)|\leq \frac{2p\sqrt p}{(p-1)}\,|x|_p\leq \frac{2p\sqrt p}{(p-1)}.
\end{equation*}
So, the Ascoli-Arzela theorem implies the existence of a convergent subsequence, and consequently the precompactness of the image of the Lipshitz ball $B_{\mathcal D}$ in $A/A'_{\mathcal D}$.
This completes the proof of the theorem.

\end{proof}

\end{document}